\documentclass[12pt]{article}
\usepackage[usenames]{color}
\usepackage{amsmath}
\usepackage{amsfonts}
\usepackage{amssymb}
\usepackage{graphicx}
 \usepackage{epsfig}

\usepackage{amsthm}
\usepackage{latexsym}
\usepackage{lscape}

\newtheorem{theorem}{Theorem}[section]

\newtheorem{definition}[theorem]{Definition}
\newtheorem{corollary}[theorem]{Corollary}
\newtheorem{proposition}[theorem]{Proposition}
\newtheorem{remark}[theorem]{Remark}
\newtheorem{example}[theorem]{Example}

\newcommand{\ignore}[1]{}

\renewcommand{\epsilon}{\varepsilon}
\newcommand{\Section}[1]{\section{#1} \setcounter{equation}{0}}
{\begingroup

 \begin{enumerate}}
 {\end{enumerate}\endgroup}

\title{Self-Similar Graphs}

\author{Kiran B. Chilakamarri\\Department of Mathematics\\
Texas Southern University\\ Houston, TX 77004 \and
M. F. Khan\\ Department of Computer Science\\Texas Southern University\\ Houston, TX 77004 \and
C. E. Larson\\Department of Mathematics and Applied Mathematics\\
Virginia Commonwealth University\\ Richmond, VA 23284\and
C. J. Tymczak\\Department of Physics\\Texas Southern University\\Houston, TX 77004}

\begin{document}

\maketitle

\begin{abstract}
For any graph $G$ on $n$ vertices and for any  {\em symmetric} subgraph $J$ of $K_{n,n}$, we construct an infinite sequence of graphs based on the pair $(G,J)$.  The First graph in the sequence is $G$, then at each stage replacing every vertex of the previous graph by a copy of $G$ and every edge of the previous graph by a copy of $J$ the new graph is constructed. We call these graphs {\em self-similar} graphs. 
We are interested in delineating those pairs $(G,J)$ for which the chromatic numbers of the graphs in the sequence are bounded. Here we have some partial results.  When $G$ is a complete graph and $J$ is a special matching we show that every graph in the resulting sequence is an {\em expander} graph.

\end{abstract}

\Section{Introduction}
\label{S:introduction}
For general graph theoretic concepts we refer the reader to any standard book on graph theory ([1],[2]). The concept of self-similarity is may be described as follows: we may say an object is self-similar if it can be broken up into pieces and all pieces appear to be same (in some sense) as the whole object except for the scale.  We say a graph  $G$ is {\em self-similar} if there is a partition of vertex set $V$ into $k$ disjoint sets $V_1,V_2,\cdots,V_k$ so that  $G[V_1] \cong G[V_2] \cong \cdots \cong G[V_k] \cong H$
where $G[U]$  denotes the sub-graph on vertex set $U$ and $H$ is the graph obtained by contracting each $V_i$  to a vertex and deleting multiple edges and the symbol $\cong$  indicates the graph isomorphism.  
In this paper we construct a large family of self-similar graphs that have a prescribed chromatic number and some classes of these graphs are shown to be expander graphs.  In Section 2 we will provide a rigorous definition of self-similar graph. In what follows we give a general description.  Let $G$  be any graph on $n$  vertices.  
Letting $G^1=G$ , we construct $G^2$  by replacing every vertex of $G^1$  by a copy of $G$ and if two vertices $x$ and $y$ are adjacent then we attach a bundle of edges $J$  (a sub-graph of the complete bi-partite graph $K_{nn}$ ) between the two copies of $G$  corresponding to vertices $x$ and $y$ and if there is no edge between $x$ and $y$ then no edges are added between the corresponding copies of $G$  .  
Since we are dealing with undirected graphs, we want the bundle of edges joining two copies of $G$ corresponding to the adjacent vertices to be a symmetric bundle. We can repeat this construction to obtain $G^3$  by replacing each vertex of $G^2$  by a copy of $G$  and placing the same edge bundle $J$ between two copies corresponding to adjacent vertices in$G^2$ .  

Repeating this process we can construct an infinite sequence of graphs $G^1,G^2,\cdots$.  Similar constructions have appeared in Physics, see [3] for instance.   We are interested in those pairs $(G,J)$  for which the infinite sequence of graphs $G^1,G^2,\cdots$   have a bounded chromatic number.   For instance, if $G$ is the complete graph on n vertices and $J$ is the complete bi-partite graph $K_{nn}$, then $G^i=K_{n^i}$, a complete graph on $n^i$ vertices ($\chi (G^i)=n^i$)   and if $G$   is any graph and $J$ has no edges, then $G^i$ is simply $n^i$ many disconnected copies of $G$   and all graphs $G^i$ in the sequence have same chromatic number namely $\chi(G)$. 
 In Section 2, we give rigorous definitions, some examples and basic properties. In Section 3, we study two special cases of $J$ for which the infinite sequence of graphs have constant chromatic number (Theorem 3.1 and Theorem 3.2). 
  In the same section we delineate those edge bundles for which chromatic numbers increase unboundedly for any non-trivial graph (Theorem 3.4) and we also characterize those edge bundles for which the chromatic numbers of $G^1,G^2,\cdots$  is bounded when $G$ is a complete graph.  In Section 4, we give a complete analysis for small complete graphs on two and three vertices.  In Section 5, we will produce an infinite sequence of edge expander graphs with fixed expansion coefficient.  In Section 6, we explore some potential applications and list some unsolved problems.

\Section{Definitions and Examples}
Throughout the rest of this paper we will write $x \sim y$ to indicate adjacency between two vertices $x$ and $y$  in a graph. Let $G$  be a simple graph with vertex set $V(G)=\{v_1,v_2,\cdots,v_n\}$.  Let $K_{nn}$ be the complete bipartite graph with bipartition $I_n$  and $ I'_n$, where $I_n=\{1,2,\cdots,n\}$  and $I'_n=\{1',2',\cdots,n'\}$ .  Let $J$ be a subgraph of $K_{nn}$  with $V(J)=I_n \cup I'_n$ , whose edges satisfy the symmetry condition $i \sim j'$  if and only if $j \sim i'$ , and we will refer to $J$  as a symmetric subgraph of $K_{nn}$ or symmetric edge bundle.   Since we attach the edge bundle between two copies of $G$, we may also write $v_i \sim v'_j $  instead of $i \sim j'$  .  Next we define recursively the infinite sequence of self-similar graphs for any pair $(G,J)$.

\begin{definition}\label{L:example}
Given a simple graph $G$  on $n$ vertices and $J$ a symmetric subgraph of $K_{nn}$,
 let $G^1=G$, $G^k$ is constructed from  $G^{k-1}$ as follows:
 \newline
$V(G^k)= \{(v_{i_1},v_{i_2},\cdots,v_{i_k}):v_{i_j}  \in V(G) \}$, and the edges are defined in terms of adjacencies in $G^{k-1}$  and adjacencies in $J$,  
 \newline
 $(v_{i_1},v_{i_2},\cdots,v_{i_k}) \sim (v_{j_1},v_{j_2},\cdots,v_{j_k}) \in G^k$ if and only if either 
 \newline
 (1)  $(v_{i_1},v_{i_2},\cdots,v_{i_{k-1}})=(v_{j_1},v_{j_2},\cdots,v_{j_{k-1}})$ and $v_{i_k} \sim v_{j_k}$ in $G$ , or
  \newline
(2) $(v_{i_1},v_{i_2},\cdots,v_{i_{k-1}}) \sim (v_{j_1},v_{j_2},\cdots,v_{j_{k-1}})$ in $G^{k-1}$   and  $i_k \sim j'_k$ in $J$.
  
\end{definition}

The adjacency condition (1) says we replace every vertex in $G^{k-1}$  with a copy of $G$   and the condition (2) says if two vertices are adjacent in $G^{k-1}$ then attach an edge bundle   between the corresponding copies of $G$.
These graphs become very large very quickly, but we can visualize these graphs for some smaller examples.

\begin{example}\label{L:example}
Suppose $G$  is a single edge, i.e., $G=K_2$, and edges of $J$  are $\{(1,1' ), (2,2')\}$ then $G^k$ is the hypercube $Q_k$.
\end{example} 

\begin{example}\label{L:example}
Suppose $G$  is a single edge, i.e., $G=K_2$, and edges of $J$  are $\{(1,2' ), (2,1')\}$ then $G^k$ is the hypercube $Q_k$.
\end{example} 

\begin{example}\label{L:example}
Suppose $G$  is a single edge, i.e., $G=K_2$ , and edges of $J$  are\\ $\{(1,1'),(1,2'),(2,1'),(2,2')\}$, then $G^k$ is the complete graph on $2^k$  vertices, i.e., $K_{2^k}$.
\end{example}

Next we will count the number of edges of $G^k$ and find  formula for the degree of vertices in $G^k$ in terms of the number of edges and degrees in $G$  and $J$.  Let $e(H)$  denote the number of edges in a graph $H$  and $d_H(x)$ denote the degree of vertex $x$ in $H$. We will simply write $d(x)$  for degree if the context is clear. 

\begin{proposition}
Let G be any simple graph on $n$  vertices and let $J$ be a symmetric sub-graph of $K_{nn}$ and $e_J=e(J)$, then $ e(G^k)=e(G) \lceil \frac{{n^k}-{e_J}^k}{n-e_J} \rceil $, if $e_{J} \neq n$ and $e(G^k)=kn^{k-1}e(G)$ if $e_J=n$.
\end{proposition}

\begin{proof}
Remembering that $G^k$ is constructed from $G^{k-1}$  by replacing each of it�s vertices by a copy of $G$  and attaching an edge bundle between the copies of $G$ corresponding to edges in $G^{k-1}$, we can write $e(G^k)= v(G^{k-1})e(G)+e(G^{k-1})e_J$ .  Since $v(G^{k-1})=n^{k-1}$,

\begin{align}
e(G^k)  &= e(G)n^{k-1}+e_j e(G^{k-1}  \nonumber )\\
  &= e(G)[n^{k-1} +n^{k-2} e_J] + {e_J}^2e(G^{k-2}) \nonumber \\
  \vdots   \nonumber  \\
  &= e(G)[n^{k-1} +n^{k-2} e_J + \cdots + n {e_J}^{k-2} + {e_J}^{k-1}]   \nonumber \\
  &= e(G) \lceil  \frac{n^k - {e_J}^k}{n- e_J}   \rceil  \text{ if } e_J \neq n    
\end{align}

\end{proof}
\begin{proposition}
Let $G$ be any simple graph on $n$  vertices and let $J$ be a symmetric sub-graph of $K_{nn}$ then, $d(u_1,u_2,\cdots,u_k)= d_G(u_k) + \sum_{i=1}^{k-1} d_G(u_{k-i}) \prod_{l=0}^{i-1} d_J(u_{k-l})$where $d(u_1,u_2,\cdots,u_k)$  is the degree of the vertex $(u_1,u_2,\cdots,u_k)$  in $G^k$.

\end{proposition}

\begin{proof}
The vertex $(u_1,u_2,\cdots,u_k)$ is adjacent to $(v_1,v_2,\cdots,v_k)$  if either\\
 (i) $ u_1=v_1, u_2=v_2, \cdots,u_{k-1}=v_{k-1}$  and $u_k \sim v_k $ in $G$  contributing $d_G(u_k)$ to  $d(u_1,u_2,\cdots,u_k)$ \\ 
 or
 \\
(ii)    $(u_1,u_2,\cdots,u_{k-1}) \sim (v_1,v_2,\cdots,v_{k-1})$ in $G^{k-1}$  and $k \sim k'$   in  $J$ contributing the product of  $d_{G^{k-1}}(u_1,u_2,\cdots,u_{k-1})$ and $ d_J (u_k)$  to $d(u_1,u_2,\cdots,u_k)$  .  Thus we have,
 \begin{equation}
 d(u_1,u_2,\cdots,u_k)=d_G(u_k)+d_{G^{k-1}}(u_1,u_2,\cdots,u_{k-1})d_J(u_k)
 \end{equation}   
and the formula for $d(u_1,u_2,\cdots,u_k)$   follows by repeated application of this recursion relation.

\end{proof}

The general formula for $d(u_1,u_2,\cdots,u_k)$ is not in closed form, but in some special cases we have closed formulas:\\
(A) If both  $G$ and $J$   are regular graphs with $d_G(u) \equiv d$ and $d_J(v) \equiv c \geq 2 $, then 
\begin{equation}
d_{G^k}(u_1,u_2,\cdots,u_k)=d[\frac{c^k-1}{c-1}].
\end{equation}
 \\
(B) If  $G$ is any graph and $J$ is a matching, i.e., $d_J(v) \equiv 1$, then

\begin{equation}
d_{G^k}(u_1,u_2,\cdots,u_k)=d_G(u_1)+d_G(u_2) + \cdots + d_G(u_k).
\end{equation} 
\\
(C) If $G$  is any regular graph of degree $d$  and $J$ is a matching, then
\begin{equation}
d_{G^k}(u_1,u_2,\cdots,u_k)=kd.
\end{equation}

\section{Chromatic Numbers of graphs  $G^k$}

In this section we will calculate $\chi (G^k)$ for some pairs  $(G,J)$ where $G$  is arbitrary but $J's$ are special edge bundles.  It is clear that  $\{\chi (G^k)\}$ is a non-decreasing integer sequence. We write $\chi_\infty (G,J)$ for the limit of the sequence  and write 
$\chi_\infty (G,J)=\infty $  if the limit is unbounded. 
\begin{theorem}
Let $G$ be any simple graph on $n$  vertices and let $J$ be a matching with edge set $E(J)={(i,i'):1 \le i \le n}$.  Then, $\chi_\infty (G,J)=\chi (G)$  for all integers $k \ge 1$ .

\end{theorem}

\begin{proof}
Let   $V(G)=\{v_1,v_2,\cdots,v_n\}$ and let $\chi (G)=p$ for some positive integer $p$.  Let $G$  be colored with group elements from the additive group of integers Modulo $p$, i.e., $Z_p=\{0,1,\cdots,p-1\}$  and let $C$ be the coloring function for $G$.  We define a coloring function $C_k :V(G^k) \longrightarrow Z_p$  inductively as follows: $C_1$ is $C$, i.e.,
 $C_1(v)=C(v)$ for all $v \in V(G)$
 Having defined a proper coloring function $C_{k-1}$  for $G^{k-1}$, we define $C_k$  by, 
 \begin{equation}
 C_k(v_1,v_2,\cdots ,v_k)=C_{k-1}(v_1,v_2,\cdots ,v_{k-1})+C(v_k) \pmod {p}
\end{equation}
  To see that $C_k$ is a proper coloring, suppose $(v_1,v_2,\cdots,v_k) \sim (u_1,u_2,\cdots,u_k) $ then either \\
  (i)  $(v_1,v_2,\cdots,v_{k-1}) =(u_1,u_2,\cdots,u_{k-1})$ and $v_k$  is adjacent to $u_k$ in $G$  or\\
(ii) $(v_1,v_2,\cdots,v_{k-1}) \sim (u_1,u_2,\cdots,u_{k-1})$ and $v_k=u_k$  (by the special choice of $J$ ).  \\
Thus,  in case (i) $C_k(v_1,v_2,\cdots,v_k)-C_k(u_1,u_2,\cdots,u_k)=C(v_k)-C(u_k)$ not zero since $v_k$ and $u_k$ are adjacent in $G$. \\ In case (ii),\\   $C_k(v_1,v_2,\cdots,v_k)-C_k(u_1,u_2,\cdots,u_k)= C_{k-1}(v_1,v_2,\cdots,v_{k-1})-C_{k-1}(u_1,u_2,\cdots,u_{k-1})$ not zero since $v_k=u_k$ and $C_{k-1}$ is a proper coloring.  Thus we have shown 
$\chi (G^k)=p=\chi(G)$ for all $k \geq 1$.  It is worth noting that the matching in the above theorem is a special matching.

\end{proof}

\begin{theorem}
Let  $ G=(V,E) $ be any finite graph with 
$ V=\{v_1,v_2,\cdots,v_n\}$  and
 let $ J = \{ (i,j'): v_i \sim v_j \textsl{ in }  G  \} $ , then $ \chi(G^k)= \chi (G) $ for all $ k \geq 1$.
 \end{theorem}

\begin{proof}
Suppose $\chi (G)=p$ and suppose $S_1,S_2,\cdots,S_p$ are the color classes of $G$.  Clearly, these sets $S_1,S_2,\cdots,S_p$  partition $V$   into independent sets.  Now consider the following partition of $V(G^k)$, 
\begin{equation}
V(G^k)=\cup_{i=1}^p  V \times V \times \cdots \times S_i .
\end{equation}
 We will show that each of $V \times V \times \cdots \times S_i$ are independent in $G^k$.  Let $(v_1,v_2,\cdots,v_{k-1},x)$  and $(u_1,u_2,\cdots,u_{k-1},y)$  be two vertices from $V \times V \times \cdots \times S_i$.  If $(v_1,v_2,\cdots,v_{k-1})= (u_1,u_2,\cdots,u_{k-1}) $, since $S_i$ is independent in $G$, $x$  and $y$  are not adjacent in G and so  $(v_1,v_2,\cdots,v_{k-1},x)$  and $(u_1,u_2,\cdots,u_{k-1},y)$ are not adjacent. \\
 On the other hand if  $(v_1,v_2,\cdots,v_{k-1})\neq (u_1,u_2,\cdots,u_{k-1}) $ and are not adjacent in $G^{k-1}$, then  $(v_1,v_2,\cdots,v_{k-1},x)$  and $(u_1,u_2,\cdots,u_{k-1},y)$ are not adjacent. Finally if $(v_1,v_2,\cdots,v_{k-1})\neq (u_1,u_2,\cdots,u_{k-1}) $ but are adjacent in $G^{k-1}$, then for an edge to exist between $(v_1,v_2,\cdots,v_{k-1},x)$  and $(u_1,u_2,\cdots,u_{k-1},y)$ it is necessary to have an edge between $x$ and $y'$ in $J$ which is impossible by the definition of $J$. This observation proves our assertion.
\end{proof}

\begin{remark}
We may note that if $\chi_\infty(G,J)$  is finite, then for every sub-graph $H$ of $G$, $\chi_\infty(H,J)$ is also finite.  Similarly, for every symmetric sub-graph $J'$   of $J$ , $\chi_\infty(G,J')$  is finite.  If $\chi_\infty(G,J)$ is infinite, then $\chi_\infty(H,J)$  and $\chi_\infty(G,J')$  are also infinite respectively for every super-graph $H$   of $G$ and for every super-graph $J'$  of $J$ .  
\end{remark}
Let $J^*$ be an edge bundle with edge set $ \{(i,j'): 1 \leq i \neq j \leq n \}$.  In other words $J^*$ is the result of removing the matching $\{ (i,i'): 1 \leq i \leq n \}$. In Theorem 3.2, suppose $G=K_n$, then $J$ is $J^*$,and we have $\chi_\infty(G,J^*) \leq n$.  This observation results in the following corollary.
\begin{corollary}
Let $G$ be any finite simple graph on n vertices, then $\chi_\infty(G,J^*) \leq n$.
\end{corollary}

Any sub-graph $J$ of $K_{nn}$   can be represented as a directed graph
 $\overline{J}$  on the vertex set $ \{1,2,\cdots,n\}$   as follows: If there is an edge from $i$  to $j'$, then we draw a directed arc from $i$   to $j$ , and if there is an edge from $i$ to $i'$, then we draw a loop at $i$ . If $J$ is a symmetric edge bundle, then $\overline{J}$   is an undirected graph with some loops. We may freely identify the vertex $v_i$  of $G$  with vertex $i$ of $\overline{J}$ .  With this definition of $\overline{J}$, we can restate all theorems we proved so far in simple manner. The Theorem 3.1 can be restated as � $\chi_\infty(G,J)=\chi(G)$ for any graph $G$   if $ \overline{J}$  is a graph of $n$  isolated loops�. The Theorem 3.2 can be restated as 
 � $ \chi_\infty(G,J)=\chi(G)$ for any graph if $\overline{J}$ is same as $G$�.   The Corollary 3.3 can be restated as
  � $ \chi_\infty(G,J) \leq n$ for any graph $G$ if $\overline{J}$  is a complete graph $K_n$  with no loops�. The next theorem says that if $G$ and $\overline{J}$  both have a common edge and  $\overline{J}$ has a loop is attached to that particular edge, then $\chi_\infty(G,J)=\infty$ .

\begin{remark}
In this article we deal with only undirected edge bundle graphs.  But all the concepts in this article can be carried over to directed graphs as well in which case starting with a directed graph $G^1=G$  each vertex of $G^{k-1}$  is replaced by a copy of $G$ and each directed arc from $x$ to $y$ of $G^{k-1}$  is replaced by an edge bundle $J$ (not necessarily symmetric) where $I_n$  is identified with the copy of $G$ corresponding to vertex $x$  and $I'_n$ identified with the copy of  $G$ corresponding to the vertex $y$ , furthermore the edge bundle can be replaced by an arc bundle.  For now however we have only the symmetric edge bundles, thus we deal with only undirected $\overline{J}$.  
\end{remark}

\begin{theorem}
Let $G$ be any simple finite graph with an edge between vertices $v_i$  and $v_j$ . If $J$ contains edges $(i,i') ,(i,j')$ and $(j,i')$ then $\chi_\infty(G,J)=\infty$.
\end{theorem}
\begin{proof}
We will show that the clique number $\omega (G^k)$ is at least $k+1$  for $k \ge 1$.  Clearly $v_i$  and $v_j$  forms a $K_2$  in $G^1=G$.  The set $S_2= \{(v_i,v_j),(v_j,v_i),v_i,v_j) \}$ forms a $K_3$  in  $G^2$ since $(v_i,v_j)$  is adjacent to $(v_i,v_j)$  by condition (1) of Definition 2.1 and $(v_j,v_i)$ is adjacent to both $(v_i,v_j),(v_i,v_j)$ and  by condition of (2) of Definition 2.1. 
Adjoining the vertex $v_i$  to each element of $S_2$  results in a $K_3$ in $G^3$ by condition (2) since $J$ contains the edge $(i,i')$.  Thus $(v_i,v_i,v_i),(v_j,v_i,v_i)$ and $(v_i,v_j,v_i)$ forms a $K_3$  in $G^3$. Furthermore all three vertices are adjacent to
$ (v_i,v_i,v_j)$.  To see this note that $(v_i,v_i,v_j)$  is adjacent to 
$(v_i,v_i,v_i)$ by condition  (1) and $(v_i,v_i,v_j)$   is adjacent to   $(v_j,v_i,v_i)$ and $(v_i,v_j,v_i)$ by condition (2) because $(v_i,v_i)$ is adjacent both $((v_j,v_i),(v_i,v_j)$  and in $G^2$.  Thus the set   $S_3=\{(v_i,v_i,v_i),(v_j,v_i,v_i),(v_i,v_j,v_i),(v_i,v_i,v_j) \}$ forms a   $K_4$ in $G^4$.  Now suppose the set 
$S_{t-1}=\{ (v_i,v_i,\cdots,v_i),(v_j,v_i,\cdots,v_i),(v_i,v_j,\cdots,v_i),\cdots,(v_i,v_i,\cdots,v_j) \}  \subset V^{t-1}$ forms a complete graph $K_{t}$ on $t$ vertices in $G^{t-1}$.  To each vertex in $S_{t-1}$  concatenating the coordinate $v_i$ results in a set of $t$ vertices in   $G^t$ which form a complete graph by condition (2).  All these vertices are adjacent to $(v_i,v_i,\cdots,v_i,v_j)$.  To see this first note that by condition (1) $(v_i,v_i,\cdots,v_i,v_j) \sim (v_i,v_i,\cdots,v_i,v_i)$  and by condition (2) it follows that 
$(v_i,v_i,\cdots,v_i,v_j) $ is adjacent to $(v_j,v_i,\cdots,v_i,v_i) ,
(v_i,v_j,\cdots,v_i,v_i),\cdots,(v_i,v_i,\cdots,v_j,v_i)$.
Thus,\\ $S_t=\{ (v_i,v_i,\cdots,v_i,v_i),(v_j,v_i,\cdots,v_i,v_j),
(v_i,v_j,\cdots,v_i,v_j),\cdots,(v_i,v_i,\cdots,v_i,v_j) \}$  \\
forms a complete graph on t+1 vertices in $G^t$. We have shown that 
$\omega (G^k) \ge k+1$ for all $k \ge 1$.

\end{proof}

\begin{theorem}
Let $G$ be a simple graph with vertex set $V(G)=\{ v_1,v_2,\cdots,v_n \}$ and
 let $J$ be a symmetric sub-graph of $K_{nn}$.  For any $i \neq j$  with 
 $i,j \in \{ 1, 2, \cdots,n \}$ , if $J$ does not contain both the edge $(i,i')$ and the pair
 $(i,j'),(j,i')$ , then $\chi_\infty(G,J) \leq 2n$.
\end{theorem}

\begin{proof}
From Remark 3.1, it is sufficient to prove the result for $G=K_n$ .  Given any symmetric  $J$ it may contain $r$  many edges of type $(i,i')$ ,$0 \leq r \leq n$.   From Remark 3.1 it is sufficient prove the theorem for maximal $J$ satisfying the condition of containing subset $(i,i')$ or ${(i,j'),(j,i')}$ but not both.  If $r=0$, then the result follows from Theorem 3.2.  If $r=n$, then the result follows from Theorem 3.1.  We need to prove the result for values of $r$  ranging from $1$   to $n-1$.  
Since $G=K_n$ , for any $r$, there is no loss of generality in assuming that $J$  contains $(1,1'),(2,2'),\cdots,(r,r')$ . The condition � $J$ contains either $(i,i')$ or$ (i,j')$ and $(j',i)$ but not both types � excludes edges of type $(i,j')$  with $1 \leq i \neq j \leq r$.  For any$i$  and $j$  with $r+1 \leq i \neq j \leq n$ , we can include edges of type $(i,j')$  in $J$.  Thus, for a given value of $r$, the edge set of the maximal edge bundle $J_r$ is given by,
\begin{equation}
J_r=\{ (i,i'):1 \leq i \leq r\} \cup \{ (i,j'): r+1 \leq i \neq \leq n \}.
\end{equation}
 
Using the graph description of $J$, the graph $\overline{J_r}$  is simply a union of a collection of $r$ isolated loops and a complete graph on $n-r$ vertices.  We will now exhibit a coloring scheme to color $G^k$ with $2n$ colors assuming that   $G^{k-1}$can be colored with $2n$  colors.  Clearly   $G^1=K_n$can be colored in $n$  colors, thus it can be colored in $2n$  colors with several empty color classes. 
Let $A_1^{k-1},A_2^{k-1},\cdots,A_{2n}^{k-1}$  be the color classes in $G^{k-1} $.  Using these color classes we now partition the vertex set of   as follows:
 \begin{equation}
 T_{p.q}=\{ (v_{i_1},v_{i_2},\cdots,v_{i_{k-1}},v_q):  (v_{i_1},v_{i_2},\cdots,v_{i_{k-1}}) \in A_p^{k-1} \}, 1 \leq p \leq 2n,1 \leq q \leq n.
 \end{equation}
 
We may note that each of these sets, i.e., $T_{p,q}$, are independent sets, but they are not maximal independent sets.  The Table 1 shows a coloring scheme.  The colors are just integers $1, 2, �, 2n$.  The color scheme is presented as a rectangular array of   $n$ rows and $2n$  columns.  The color assigned in the entry corresponding to the $i^{th}$  row and $j^{th}$ column corresponds to the color assigned to the vertices in the set $T_{j,i}$  where $1 \leq i \leq n,1 \leq j \leq 2n$.  We use colors  $1,2,\cdots,2n$ in the first row, and then fill subsequent $r-1$  rows by rotating the colors by two units at a time, i.e., the second row uses colors $3,4,\cdots,2n,1,2$   and so on. Thus the first $r$ rows have permutations of all colors so that no two elements in the same row are same. We may note that the effect of shifting by two units at a time is that the first $r$ colors in the odd numbered columns are odd and the first $r$  colors in the even numbered columns are even. The rest of the colors in all odd numbered columns are even numbers from $2$ to $2(n-r)$ and the colors in all even numbered columns are odd numbers from $1$ to $2(n-r)-1$ .
\\
\begin{table}[h]\footnotesize
  \caption{Coloring Scheme}
\begin{tabular}{|c|c|c|c|c|c|c|c|}
\hline
.& $A_1^{k-1}$&$A_2^{k-1}$&$A_3^{k-1}$&$A_4^{k-1}$&$\cdots$&$A_{2n-1}^{k-1}$ & $A_{2n}^{k-1}$\\
\hline
$v_1$ & 1& 2& 3& 4& $\cdots$ & $2n-1$ & $2n$\\
\hline
$v_2$ & 3& 4& 5& 6& $\cdots$ & $1$ & $2$\\
\hline
$v_3$ & 5& 6& 7& 8& $\cdots$ & $3$ & $4$\\
\hline
$v_4$ & 7& 8& 9& 10& $\cdots$ & $5$ & $6$\\
\hline
$ \vdots$ & $ \vdots$& $ \vdots$& $ \vdots$& \vdots & $\vdots$ & $$ \vdots$$ & $$ \vdots$$\\
\hline
$v_r$ & $2r-1$& $2r$& & & $\cdots$ &  & \\
\hline
$v_{r+1}$ & 2& 1& 2& 1& $\cdots$ & $2$ & $1$\\
\hline
$v_{r+2}$ & 4& 3& 4& 3& $\cdots$ & $4$ & $3$\\
\hline
$ \vdots$ & $ \vdots$& $ \vdots$& $ \vdots$&  $\vdots$ & $$ \vdots$$ & $$ \vdots$$ & \vdots$$\\
\hline
$v_{n-1}$ & $2(n-r-1)$& $2(n-r)-3$& $2(n-r-1)$& $2(n-r)-3$& $\cdots$ & $2(n-r-1)$ & $2(n-r)-3$\\
\hline
$v_n$ & $2(n-r)$& $2(n-r)-1$& $2(n-r)$& $2(n-r)-1$& $\cdots$ & $2(n-r)$ & $2(n-r)-1$\\
\hline
\end{tabular}
\end{table}
\\
\\
To verify that the coloring scheme is a proper coloring (not necessarily optimal coloring) let us consider two vertices 
$\alpha =(v_{i_1},v_{i_2},\cdots,v_{i_k}) $ and $\beta =(u_{j_1},u_{j_2},\cdots,u_{j-k})$     in $G^k$.  There is an edge between $\alpha$  and $\beta$  if either (i) $(v_{i_1},v_{i_2},\cdots,v_{i_{k-1}}) = (u_{j_1},u_{j_2},\cdots,u_{j-k})$   and $v_{i_k} \sim v_{j_k}$ in $G$  or 
(ii)   $(v_{i_1},v_{i_2},\cdots,v_{i_{k-1}}) \sim (u_{j_1},u_{j_2},\cdots,u_{j-k})$ in $G^{k-1}$  and $v_{i_k} \sim u_{j_k}$ in $J_r$.  Since the equality of first $k-1$  coordinates is possibly only if both vertices $\alpha$ and $\beta$ are in the same column, the case (i) is possible along columns, and since $G=K_n$  there are possible edges between two sets in the same column.  Thus to ensure proper coloring each column must use distinct colors. This is true in the coloring scheme in Table 3.1 since the first $r$  colors in each column are consequence of permutations of $2n$  colors, and are either all even or all odd, while the rest of $n-r$  colors in each column are distinct and opposite in parity (to the first $r$ colors). This leaves us with case (ii), then $(v_{i_1},v_{i_2},\cdots,v_{i_{k-1}}) \sim (u_{j_1},u_{j_2},\cdots,u_{j_{k-1}})$  in $G^{k-1}$ and $v_{i_k} \sim u_{j_k}$ in $J_r$.  In this case $\alpha$  and $\beta $ must be in different columns.  If $1 \leq i_k \leq r$, then an edge between $\alpha$ and $\beta$ is possible only if $v_{i_k} =u_{j_k}$ , i.e., only if $\alpha$ and $\beta $ are in the same row in which they have different colors since the first $r$  rows are permutations.  If $r+1 \leq i_k \leq n$, there is no edge between $\alpha$ and $\beta$  if both $\alpha$ and $\beta$ are in the same row.  Thus we can use same color within in row for rows $r+1,r+2,\cdots,n$   and distinct colors for distinct rows.  This is true since we use a pair of colors $2i-1,2i$  for the row $r+1$  for1 $ \leq i \leq n-r$ . This completes the proof.

\end{proof}
Combining Theorem 3.4 and Theorem 3.5 we have shown the following result for complete graphs.

\begin{theorem}
$\chi_\infty(K_n,J)$  is finite if and only if $J$ does not contain the edge triple $(i,i'),(i,j')$, $(j,i')$   and for any $i \neq j $  with $i,j \in \{1,2,\cdots,n \}$.
Using the graph description of $J$, we may restate this as: $\chi_\infty(K_n,J)$  is finite if and only if the graph 
$\overline{J}$  on $n$ vertices is a union of isolated loops and a complete graph.

\end{theorem}
\section{Complete Analysis for $K_2$ and $K_3$}
From Remark 3.1 to make a complete analysis for a graph $G$, it is convenient to consider the partial order of all possible symmetric edge bundles under sub-graph containment.  For $K_2$  this partial order $\Im_2$ contain eight elements, which are described below,
 $J_1=\emptyset, J_2=\{(1,1')\} ,J_3=\{(2,2') \},J_4=\{(2,2'),(1,1') \}  ,  J_5=\{(1,2'),(2,1')\},J_6=\{(1,2'),(2,1'),(1,1')\}
 J_7=\{(1,2'),(2,1'),(2,2') \},J_8=\{(1,1'),(1,2'),(2,1'),(2,2') \} $.
For each $J$ , we will calculate $\chi_\infty(G,J)$ .\\
1.  For  $J_1=\emptyset$, $G^k$  is simply $2^{k-1}$  many edges and $\chi_\infty(K_2,J_1)=2$ .\\
2. For $J_2=\{(1,1')\}$ or $\{(2,2')\} $, $G^k$ is a tree on $2^k$ vertices.  We can describe these trees as follows:   $G^1=K_2$ is an edge, having constructed $G^{k-1}$ on $2^{k-1}$ vertices, attaching an edge at each vertex of  $G^{k-1}$, a total of $2^{k-1}$ edges, results in $G^k$.  It is easy to see $G^k$ is tree and $\chi_\infty(K_2,J_2)=\chi_\infty(K_2,J_3)=2$.\\ 
3. For $J_4=\{(2,2'),(1,1')\},J_5=\{(1,2'),(2,1')\}$ it is easy to see that $G^k$  is   the  $k$-dimensional hypercube $Q_k$.  Since hypercube is a bipartite graph we have $\chi_\infty(K_2,J_4)=\chi_\infty(K_2,J_5)=2$.\\
4.  For $J_6=\{(1,2'),(2,1'),(1,1')\} $ or  $J_7=\{(1,2'),(2,1'),(2,2') \}$, by Theorem 3.4 we have  
 $\chi_\infty(K_2,J_6)= \infty$ and $\chi_\infty(K_2,J_7)=\infty$. \\
5.  For $J_8=\{(1,1'),(1,2'),(2,1'),(2,2') \}$, $G^k=K_{2^k}$,  a complete graph on $2^k$  vertices and $\chi_\infty(K_2,J_8)=\infty$ .

In case of $K_3$, the partial order $\Im_3$   has too many elements if we use labeled edge bundles, and has twenty distinct elements if we use unlabeled edge bundles.  Using graph description of edge bundles $\chi_\infty(K_3,J)=\infty$  if $\overline{J}$  has a loop attached to an edge. Then we are left with three maximal edge bundles for which $\chi_\infty(K_3,J)$  is finite.  The graphs $\overline{J}$ for these three edge bundles can be described as follows:\\
 (i) $\overline{J}$ is three isolated loops\\
  (ii) one isolated loop and an edge or\\
  (iii) a complete graph on three vertices.  \\
  Theorem 3.1 is applicable in case (i) and  Corollary 3.3 is applicable in case (iii) and in both cases we $\chi_\infty(K_3,J)=3$ conclude .  In case (ii) Theorem 3.5 suggests $\chi_\infty(K_3,J) \leq 6$ .  We will now prove 
   $\chi_\infty(K_3,J)=4$  in case (ii).

\begin{theorem}
 If $J=\{(1,1,'),(2,3'),(3,2') \}$ , then $\chi_\infty(K_3,J)=4$ .
\end{theorem}

\begin{proof}
We will first prove three colors are not sufficient for $G^2$  and four colors will suffice, then we will prove theorem by recursively constructing independent sets in $G^k$  for all $k \geq 3$.  Let   $V(G)= \{v_1,v_2,v_3 \}$.  Since $G=K_3$, $v_i \sim v_j$   for $i \neq j $.  In $G^2$ the vertex set can be partitioned in to three groups $\{(v_1,v_1),(v_1,v_2),(v_1,v_3) \}$, $\{(v_2,v_1),(v_2,v_2),(v_2,v_3) \}$, and   
$\{(v_3,v_1),(v_2,v_1),(v_3,v_1) \}$  where induced sub-graph of each group is a $K_3$.  Three vertices $(v_1,v_1)$ ,$(v_2,v_1)$  and $(v_3,v_1)$ one from each group forms a $K_3$  since $J$ contains the edge $(1,1')$.  Suppose $G^2$ can be colored in three colors $a$,$b$, and $c$.  Without loss of generality let
$(v_1,v_1)$, $(v_2,v_1)$, and $(v_3,v_1)$ be assigned colors $a$, $b$, and $c$ respectively.  The vertices $(v_1,v_2)$ and   
$(v_1,v_3)$ must be assigned $b$ and $c$.  Suppose color $b$ is assigned to $(v_1,v_2)$ and suppose color $c$ is assigned to $(v_1,v_3)$. Then since $(v_3,v_3) \sim (v_3,v_1)$ with color $c$ and $(v_3,v_3) \sim (v_1,v_2)$ with color $b$, the color $a$ must be used for $(v_3,v_3)$  . Now the vertex $(v_2,v_2)$  is adjacent to $(v_3,v_3)$,$(v_2,v_1)$ and $(v_1,v_3)$  with all three colors.  So, the vertex $(v_2,v_2)$  requires a fourth color.  If the color $c$ is assigned to $(v_1,v_2)$ and color $b$ is assigned to $(v_1,v_3)$ , then since $(v_2,v_3) \sim (v_1,v_2)$  and $(v_2,v_3) \sim (v_2,v_1)$, the color $a$ must be assigned to $(v_2,v_3)$.  Then $(v_3,v_2)$  requires fourth color since 
$(v_3,v_2)$   is adjacent to vertices $(v_2,v_3)$, $(v_1,v_3)$  and $(v_3,v_1)$.  Thus we need at least four colors. For the rest of this proof we will use numbers $1, 2$, and $3$ to indicate the vertices $v_1,v_2$, and $v_3$  respectively.  It is easy to verify the following four independent sets partition of $V(G^2)$  proving 
$\chi_\infty(G^2,J)=4$:\\
$A_2=\{(1,1)\},B_2=\{(2,2),(3,1)\},C_2=\{(1,2),(2,1),(3,2)\},D_2=\{(1,3),(2,3),(3,3)\}$.\\
   
We will use induction to prove $\chi(G^k,J)=4$ for all $k \geq 2$. We have shown the result for $k=2$.  Using numbers $1, 2$, and $3$ to indicate the vertices $v_1,v_2$ and $v_3$, $V(G^{k-1}$  is the set product $I_3^{k-1}$ , where 
$I_3=\{1,2,3\}$ .  Suppose $\chi(G^{k-1},J)=4$  then $I_3^{k-1}$  is partitioned in to independent sets $A_{k-1},B_{k-1},C_{k-1}, D_{k-1}$ in $G^{k-1}$.  We partition $I_3^k$  into four independent sets in  as follows:\\
$A_k=\\
\{(a_1,\cdots,a_{k-1},2): (a_1,\cdots,a_{k-1}) \in I_3^{k-1}-B_{k-1}\} \cup \{(b_1,\cdots,b_{k-1},1):(b_1,\cdots,b_{k-1}) \in B_{k-1} \}$ \\
$B_k=\\
\{(a_1,\cdots,a_{k-1},1): (a_1,\cdots,a_{k-1}) \in I_3^{k-1}-A_{k-1} \}\cup \{(b_1,\cdots,b_{k-1},1):(b_1,\cdots,b_{k-1}) \in A_{k-1} \}$\\
$C_k=\{(a_1,\cdots,a_{k-1},1): (a_1,\cdots,a_{k-1}) \in C_{k-1}\} \cup \{(b_1,\cdots,b_{k-1},1):(b_1,\cdots,b_{k-1}) \in B_{k-1} \}$\\
$D_k=\{(a_1,\cdots,a_{k-1},1): (a_1,\cdots,a_{k-1}) \in D_{k-1}\} \cup \{(a_1,\cdots,a_{k-1},3):(a_1,\cdots,a_{k-1}) \in A_{k-1} \}$\\
 \\ 
To show $A_k$ is independent we first note that there are no edges in the set\\
 $\{(a_1,a_2,\cdots,2):(a_1,\cdots,a_{k-1}) \in I_3^{k-1}-B_{k-1} \}$    since the last coordinate is 2. There are no edges in the set 
$\{(b_1,b_2,\cdots,1):(b_1,\cdots,b_{k-1}) \in B_{k-1} \}$  since $B_{k-1}$  is an independent set in $G^{k-1}$.  There are no edges across these two sets since the last coordinate in one set is $1$ and the last coordinate in the other set is $2$.  Similar arguments show the rest of the sets are independent. This completes the proof.

\end{proof}

\section{Expansion Property}
The expansion property is crucial in many applications in communication networks and this property is particularly important to build non-blocking networks, see [4] for an excellent discussion.  However, it is important to note that in our definition of spectrum of a graph we simply mean the eigenvalues of the adjacency of a graph [6], but in [4] the spectrum refers to the eigenvalues of the Laplacian of the graph.  Definitions in this section are from �Combinatorial Problems and Exercises� by Laszlo Lovasz [5].\\

The {\em Conducatnce} of a graph $G$, $\Phi(G)$, is the minimum of $\frac{\delta_G(S)}{|S|}$ over all non-empty subsets of $V(G)$ with $ |S| \leq \frac{|V(G)|}{2}$, whetre $\delta_G(S)$ is the total number of edges joining the set $S$ to it's compliment $V(G)-S$. Graphs for which the conductance bounded from below by a positive constant are called {\em expanders}.   We may refer to $\Phi(G)$ as the {\em edge expansion coefficient}.  Similarly, we may define {\em vertex expansion coefficient} for regular graphs by taking $\delta_G(S)$  as the total number vertices in $V(G)-S$  joining to $S$.  It can be shown (see exercise 31 in section 11 in [5]) that for a regular graph with $G$  with degree $d$  that 
$\Phi(G) \geq \frac{\lambda_1 - \lambda_2}{2}$ and  
$\Phi(G) \leq 2 \sqrt{d(\lambda_1 - \lambda_2)}$ where $\lambda_1$  and $\lambda_2$  are the first and second largest eigenvalues of the adjacency matrix $A(G)$  of $G$.  If $G=K_n$, and $\overline{J}$   is $n$ loops with no edges, the self-similar graphs $\{G^k \}$   are all edge expander graphs with edge expansion coefficient greater than or equal to $\frac{n}{2}$. This is the main point of the next theorem.  The results stated in following remark will be used in the proof of the theorem.

\begin{remark}
If $A=((a_{ij}))_{n \times n}$ is a matrix with all diagonal elements equal to $a$  and all non-diagonal elements equal to $b$, then 
 $det(A)=(a-b)^{n-1} [a+(n-1)b]$ .  Similarly if $A=((B_{ij}))_{nk \times nk}$  where $B_{ij}$ is a square matrix of order $k$, and  $B_{ii}=D_{k \times k}, 1\leq n$ and $B_{ij}=E_{k \times k}$   for $i \neq j$ for some matrices $D$ and $E$, then 
 $det(A)=[det(D-E)]^{n-1} [det(D+(n-1)E]$  (see [7]).
\end{remark}

\begin{theorem}
Let $G=K_n$  and let $J$  be a matching with edge set 
$E(G)=\{(i,i'): 1 \leq i \leq n\}$. Then $\Phi(G^k) \geq \frac{n}{2}$ for $k \geq 1$.
\end{theorem}
\begin{proof}
Let $V(G)=\{ v_1,v_2,\cdots,v_n \}$.  We will compute the spectrum of $G^k$  for $k \geq 1$  and then calculate the difference.  We Claim that the spectrum of $G^k$  is $\{-k,-k+n,-k+2n,\cdots,-k+kn\}$ for $k \geq 1$. We will prove this claim using mathematical induction.  
Since $G^1=K_n, A(G^1)=((a_{ij}))_{n \times n}$, where
 $a_{ij}=1$ if $ 1 \leq i \neq j \leq n $  and   $a_{ii}=0 $ for $1 \leq i \leq n$.  The characteristic polynomial of $G^1$  from Remark 5.1 is $ det(A(G^1)- \lambda I_n)=(- \lambda -1)^{n-1} (- \lambda +n-1)$.  This shows the spectrum of $G^1$  is $\{-1,n-1 \}$ , thus $\Phi(G^1) \geq \frac{n}{2}$.  
Assume that the spectrum of $G^{k-1}$  is 
 $\{(-(k-1),-(k-1)+n,\cdots,-(k-1)+(k-1)n \}$.  
 We will partition the vertices of $G^k$  into $n$ subsets $H_i$ as follows: \\
 $H_i=\{(x_1,x_2,\cdots,x_{k-1},v_i):(x_1,x_2,\cdots,x_{k-1}) \in V(G^{k-1}) \}$  for $i=1,2,\cdots,n$.  \\
 Since $i \sim i'$  in $J$, it follows that the induced sub-graph $[H_i]$  of $G^k$  is isomorphic to $G^{k-1}$, 
 $[H_i] \cong G^{k-1}$, for $1 \leq i \leq n$.  
 Thus $A(H_i)=A(G^{k-1})$ for $1 \leq i \leq n$.  For $i \neq j $, a vertex in $H_i$  is adjacent to another vertex $H_j$  if and only if both vertices have identical first $k-1$ coordinates.  This shows that $A(G^k)=((B_{ij}))_{n \times n}$  where $B_{ii}=A(G^{k-1})$  and $B_{ij}=I_{n^{k-1}}$  if $i \neq j$. From Remark 5.1, we have \\
 
 $det(A(G^k)- \lambda I_{n^k}) = (det[A(G^{k-1})-(\lambda+1)I_{n^{k-1}}])^{n-1} det[A(G^{k-1})-(\lambda+1-n) I_{n^{k-1}}]$.\\
Thus eigenvalues $\lambda$ of $A(G^k)$  are given by \\
 \begin{equation}
 \lambda +1 =-(k-1),-(k-1)+n,\cdots,-(k-1)+(k-1)n
 \end{equation}
  or
 \begin{equation}
 \lambda+1-n= -(k-1),-(k-1)+n,\cdots,-(k-1)+(k-1)n
 \end{equation}
Thus $\lambda=-k,-k+n,\cdots,-k+nk $   and \\
 \begin{equation}
 \Phi(G^k) \geq \frac{\lambda_1- \lambda_2}{2}
  =\frac{(-k+nk)-(-k+(n-1)k)}{2}=\frac{n}{2}
 \end{equation}
 
\end{proof}

\begin{corollary}
Let $G=K_n$  and let $J$ be a matching with edge set
 $E(J)=\{(i,i'):1 \leq i \leq n \}$.  Then the vertex expansion coefficient of $G^k$  is greater than $\frac{1}{2k}$  for $k \geq 1$.
\end{corollary}

\begin{proof}
Since $G^k$  is regular graph of degree $(n-1)k$, and the edge expansion coefficient $\Phi(G^k) \geq \frac{n}{2}$ , it follows that vertex expansion coefficient is at least $\frac{\Phi(G^k)}{(n-1)k}$  and  $\frac{\Phi(G^k)}{(n-1)k} \geq \frac{n/2}{(n-1)k} > \frac{1}{2k}$.
\end{proof}

At the other extreme, suppose $G=K_n$, and $J=J^*$, where $J^*$ is $K_{nn}$ with matching missing, i.e.,
$E(J^*)= \{(i,j'): 1 \leq i \neq j \neq n\}$.

Again partitioning $V(G^k)$ the   as in Theorem 5.1 it is easy to see that $A(G^k)=((B_{i,j}))_{n \times n}$  where $B_{ii}$  is a zero matrix of order $n^{k-1}$  and $B_{ij}=A(G^{k-1})$  if $i \neq j$. Solving for the characteristic equation $det(A(G^k)- \lambda I_{n^k})=0$  inductively we can show that the spectrum of   is as follows:\\
$(n-1)^k,(-1)(n-1)^{k-1},(-1)^2(n-1)^{k-2},\cdots,(-1)^{k-1}(n-1),(-1)^k$.
\\ 
Thus $\Phi(G^k) \geq \frac{(n-1)^{k-2}(n^2-2n)}{2}$.

\section{Conclusions and Problems}
For any simple undirected graph $G$, we have constructed several infinite families of graphs (which we call self-similar graphs) all of which have chromatic numbers bounded by twice the number of vertices of $G$.  For a complete graph on $N$ vertices, using a special edge bundle $J$, we have constructed an infinite family of self-similar graphs all of which have the edge expansion coefficient $\Phi$  is bounded below by $\frac{N}{2}$.  This investigation of infinite families of self-similar graphs based on a pair $(G,J)$ leads to several interesting questions.  Given a graph $G$, we found two special edge bundles for which $\chi_\infty(G,J)=\chi(G)$.  In Theorem 3.6 we have characterized the edge bundles $J$ for which $\chi_\infty(K_n,J)<\infty$ .  We believe a similar characterization exists for an arbitrary graph, hence we ask following questions.\\
Problem 1: For any graph $G$, characterize the edge bundles $J$ for which $\chi_\infty(G,J)<\infty$.\\
The next problem attempts to generalize Theorem 5.1.\\
Problem 2: Let $G$ be any graph on $N$ vertices and let $J$ be a matching with edge set $E(J)=\{(i,i'):1 \leq i \leq N \}$. 
 Let $\{G^k \}$ be the sequence of self-similar graphs based on the pair $(G,J)$.  Is it true that  $\Phi(G^k) \geq \Phi(G)$ ?
\\
\\
\
\section*{Acknowledgement}
The authors would like to acknowledge the support given by the NSF CREST CRCN center (grant HRD-1137732), and Dr. Tymczak would like to acknowledge the Texas Southern University High Performance Computing Center (http:/hpcc.tsu.edu/; grant PHY-1126251).

\end{document}